\documentclass[a4paper,12pt]{amsart}
\usepackage{amssymb}

\newtheorem{theorem}{Theorem}
\newtheorem{lemma}{Lemma}
\newtheorem{proposition}{Proposition}
\newtheorem{corollary}{Corollary}
\newtheorem*{LST}{Local Structure Theorem}
\theoremstyle{definition}
\newtheorem{definition}{Definition}
\theoremstyle{remark}
\newtheorem{remark}{Remark}

\newcommand{\AAA}{\mathbb{A}}
\newcommand{\ab}{\mathfrak{a}}
\newcommand{\ann}{\perp}

\newcommand{\Cc}{\mathcal{C}}
\newcommand{\codim}{\operatorname{codim}}
\newcommand{\cork}{\operatorname{cork}}
\newcommand{\Dd}{\mathcal{D}}
\newcommand{\df}{\operatorname{def}}
\newcommand{\g}{\mathfrak{g}}
\newcommand{\GL}{\mathrm{GL}}
\newcommand{\gr}{\operatorname{gr}}
\newcommand{\Ho}{\mathrm{H}}
\newcommand{\Ii}{\mathcal{I}}
\renewcommand{\Im}{\operatorname{Im}}
\newcommand{\Ker}{\operatorname{Ker}}
\newcommand{\kk}{\Bbbk}
\newcommand{\lv}{\mathfrak{l}}
\newcommand{\m}{\mathfrak{m}}
\newcommand{\Mo}{\oo{M}}
\newcommand{\N}{\mathfrak{N}}
\newcommand{\Nn}{\mathcal{N}}
\newcommand{\Oo}{\mathcal{O}}
\newcommand{\oo}[1]{\mathaccent"7017{#1}}
\newcommand{\p}{\mathfrak{p}}
\newcommand{\PP}{\mathbb{P}}
\newcommand{\q}{\mathfrak{q}}
\newcommand{\rk}{\operatorname{rk}}
\newcommand{\Ru}[2][\relax]{{#2}_{\mathrm{u}}\ifx#1\relax\else{\mspace{-7.8mu}}^{#1}\fi}
\newcommand{\SL}{\mathrm{SL}}
\newcommand{\So}{\oo{S}}
\newcommand{\sort}{\angle}
\newcommand{\Spec}{\operatorname{Spec}}
\newcommand{\Sym}{\mathrm{S}}
\newcommand{\tr}{\operatorname{tr}}
\newcommand{\Tt}{\mathcal{T}}
\newcommand{\un}{\mathfrak{u}}
\newcommand{\Uu}{\mathcal{U}}
\newcommand{\Vv}{\mathcal{V}}
\newcommand{\WL}[1]{\Lambda(#1)}
\newcommand{\Ww}{\mathcal{W}}
\newcommand{\Xo}{\oo{X}}
\newcommand{\Yo}{\oo{Y}}
\newcommand{\z}{\mathfrak{z}}
\newcommand{\ZZ}{\mathbb{Z}}
\newcommand{\Zz}{\mathcal{Z}}

\begin{document}

\title[Actions with invariant Lagrangian subvarieties]{Hamiltonian actions on symplectic varieties with invariant Lagrangian subvarieties}

\author{Dmitry A.~Timashev}
\address{
Lomonosov Moscow State University\\
Faculty of Mechanics and Mathematics\\
Department of Higher Algebra\\
119991 Moscow, Russia}
\email{timashev@mccme.ru}

\author{Vladimir S.~Zhgoon}
\address{
Science Research Institute of System Studies\\
Department of Mathematical Problems in Informatics\\
Nakhimovskii prospect 36-1\\
117218 Moscow, Russia}
\email{zhgoon@mail.ru}



\date{September 24, 2011}

\begin{abstract}
We prove several results on symplectic varieties with a Hamiltonian action of a reductive group having invariant Lagrangian subvarieties. Our main result states that the images of the moment maps of a Hamiltonian variety and of the cotangent bundle over an invariant Lagrangian subvariety coincide. This implies that the complexity and rank of the Lagrangian subvariety are equal to the half of the corank and to the defect of the Hamiltonian variety, respectively. This result generalizes a theorem of Panyushev on the complexity and rank of a conormal bundle. A simple elementary proof of this theorem is also given in the paper.
A generalization of the above results to some special class of invariant coisotropic subvarieties is obtained.
\end{abstract}

\maketitle

\section*{Introduction}

The aim of this paper is to study Hamiltonian actions of reductive groups on symplectic algebraic varieties with invariant Lagrangian subvarieties. Most natural examples are obtained by taking any smooth algebraic variety $X$ acted on by a reductive group $G$ and considering the cotangent bundle $T^*X$ with the induced $G$-action and canonical symplectic form; here the zero section is an invariant Lagrangian subvariety. In fact, any symplectic variety in a neighborhood of a Lagrangian subvariety has the same structure (at least analytically) as a cotangent bundle in a neighborhood of its zero section. However in the equivariant setting the situation is more delicate. Nevertheless we shall show that some important invariants of a Hamiltonian symplectic variety with an invariant Lagrangian subvariety coincide with those of the cotangent bundle over this subvariety. The main argument is deformation to the normal bundle of the subvariety.

Now we describe the results of this paper in more details.

In the theory of reductive group actions on algebraic varieties, there are several numerical invariants, which control, in particular, the harmonic analysis and equivariant compactification theory of a $G$-variety. The \emph{complexity} of an irreducible $G$-variety $X$ is the codimension of a general orbit of a Borel subgroup $B\subset G$. The set of eigenweights of all $B$-semi-invariant rational functions on $X$ forms a lattice, which is called the \emph{weight lattice} of~$X$. Its rank is called the \emph{rank} of~$X$. Now suppose that $X$ is smooth and consider a smooth $G$-subvariety $Y\subset X$. Note that the conormal bundle of $Y$ is an invariant Lagrangian subvariety in~$T^*X$. Panyushev \cite{conormal} proved that the complexities and ranks of the normal and conormal bundles of $Y$ are equal to those of~$X$. His proof was based on his theory of doubled actions and advanced technique of modern invariant theory (Luna's {\'e}tale slice theorem). We give a short elementary proof of his result using deformation to the normal bundle.

It is natural to try to generalize these results to invariant Lagrangian subvarieties other than conormal bundles. In cotangent bundles, the general case is easily reduced to the case of conormal bundles. But for arbitrary Hamiltonian $G$-varieties even the formulation of such a generalization is not obvious (since there is no~$X$!). However it was noticed by Vinberg \cite{comm&coiso} that the complexity and rank of $X$ can be expressed in terms of equivariant symplectic geometry of~$T^*X$. Namely the doubled complexity of $X$ equals the \emph{corank} of~$T^*X$ (which is the rank of the symplectic form restricted to the skew-orthocomplement of the tangent space of a general $G$-orbit) and the rank of $X$ equals the \emph{defect} of~$T^*X$ (which is the dimension of the kernel of the symplectic form on the tangent space of a general orbit). Thus one may conjecture that the complexity and rank of an invariant Lagrangian subvariety $S$ in a Hamiltonian $G$-variety $M$ are equal to the half of the corank and to the defect of~$M$, respectively. We prove this conjecture and even more: we show that (the closures of) the images of the moment maps for $M$ and $T^*S$ coincide. (It is well-known, due to Knop \cite{Weyl&mom} and Vinberg~\cite{comm&coiso}, that the above numerical invariants can be extracted from the image of the moment map.)

Finally, we extend the above results to invariant coisotropic subvarieties with invariant degeneration leaves.

We also prove that a subvariety in the zero fiber of the moment map for an affine Hamiltonian variety which consists of points whose orbit closures intersect a given invariant isotropic subvariety (a kind of nullcone) is isotropic.

\subsection*{Acknowledgements}

We are grateful to A.~V.~Petukhov for attracting our attention to the results of Panyushev. The formulation of Proposition~\ref{nullcone} in the special case of a cotangent bundle and its zero section is due to him. We thank S.~O.~Gorchinsky for useful discussions which lead to the formulation and proof of Lemma~\ref{sec.cone}. We express our gratitude to E.~B.~Vinberg for helpful comments on the preliminary version of the paper.

\subsection*{Notation and conventions}

We shall work over an algebraically closed ground field $\kk$ of characteristic zero. Algebraic varieties are assumed to be irreducible. $\kk[X]=\Ho^0(X,\Oo_X)$ denotes the algebra of regular functions on a variety $X$ and $\kk(X)$ is the field of rational functions. For smooth~$X$, $TX$ and $T^*X$ denote the tangent and cotangent bundle of~$X$, respectively.

Algebraic groups will be denoted by capital Latin letters and their Lie algebras by the respective lowercase Gothic letters. The identity component of an algebraic group $H$ is denoted by~$H^{\circ}$. Throughout the paper, $G$~denotes a connected reductive algebraic group. We denote by $B$ a Borel subgroup in~$G$, by $T$ a maximal torus in~$B$, and by $U$ the unipotent radical of~$B$. A parabolic subgroup containing $B$ is usually denoted by~$P$, with the standard Levi subgroup $L\supset T$ and the unipotent radical~$\Ru{P}$, and $P^{-}$ is the opposite parabolic intersecting $P$ in~$L$.

The Weyl group $W=N_G(T)/T$ is generated by simple root reflections corresponding to the choice of~$B\supset T$. Let $w_0$ denote the longest element of $W$ with respect to this set of generators. $W$~acts on the weights of $T$ and $\lambda\mapsto\lambda^*=-w_0\lambda$ is a linear involution of the weight lattice. If $\lambda$ is the highest weight of a simple $G$-module, then $\lambda^*$ is the highest weight of the dual module.

\section{Normal and conormal bundles}

\subsection{}

Let $X$ be a $G$-variety. Recall the following definitions.

\begin{definition}
The \emph{complexity} $c(X)$ is the codimension of general $B$-orbits in~$X$. It can also be defined as the minimal codimension of $B$-orbits in $X$ and, by the Rosenlicht theorem \cite[2.3]{inv}, coincides with the transcendence degree (over~$\kk$) of the field $\kk(X)^B$ of $B$-invariant rational functions.

The \emph{weight lattice} of $X$ is the set $\WL{X}$ of eigenweights of all (nonzero) $B$-semi-invariant rational functions on~$X$. It is a sublattice in the character lattice of~$B$ (or of~$T$).

The \emph{rank} of $X$ is $r(X)=\rk\WL{X}$.

\end{definition}

These invariants play an important role in the equivariant geometry of $X$ and related representation theory, see e.g.~\cite{hom&emb}.

\subsection{}\label{LST&N*}

For a smooth variety $X$ and a smooth (locally closed) subvariety $Y\subset X$, denote by $N(X/Y)$ and $N^*(X/Y)$ the normal and conormal bundle of $Y$ in~$X$, respectively. In \cite[\S2]{conormal} Panyushev proved the following

\begin{theorem}\label{norm&conorm}
Let $X$ be a smooth $G$-variety and $Y\subset X$ be a smooth $G$-subvariety. Put $N=N(X/Y)$ and $N^*=N^*(X/Y)$. Then
\begin{align*}
    c(X)&=c(N)=c(N^*), \\
    r(X)&=r(N)=r(N^*), \\
    \WL{X}&=\WL{N},
\end{align*}
while $\WL{N^*}$ is obtained from $\WL{N}$ by a linear transformation of the character lattice.
\end{theorem}

We shall give a short elementary proof of this theorem. First, we reduce everything to the affine case. This reduction is based on a well-known theorem, going back to Brion--Luna--Vust and Grosshans, see e.g.\ \cite[4.2]{hom&emb}:

\begin{LST}
Given a normal $G$-variety $X$ with a $G$-stable subvariety $Y$, let $P\supset B$ be the normalizer of a general $B$-orbit in~$Y$, with the standard Levi decomposition $P=\Ru{P}\leftthreetimes L$. There exist a $P$-stable affine open subvariety $\Xo\subset X$ and an $L$-stable closed subvariety $Z\subset\Xo$ such that:
\begin{enumerate}
  \item $\Xo=PZ \simeq \Ru{P}\times Z$, where $\Ru{P}$ acts on the first factor of the r.h.s.\ by left multiplication, and $L$ acts on the first factor by conjugation and on the second factor in the natural way;
  \item $\Yo=Y\cap\Xo \simeq \Ru{P}\times(Y\cap Z)$ is nonempty and closed in~$\Xo$;
  \item the $L$-action on $Y\cap Z$ amounts to a free action of a torus $A=L/L_0$, where $L \supset L_0 \supset [L,L]$.
\end{enumerate}
\end{LST}

\begin{remark}\label{loc.str.gen}
In the particular case $Y=X$ we immediately see that $\WL{X}$ is the character lattice of~$A$, $r(X)=\dim A$, and $c(X)=\dim Z-\dim A=\dim X-\dim\Ru{P}-\dim A$.
\end{remark}

If $X$ and $Y$ are smooth, then $N(\Xo/\Yo)\simeq\Ru{P}\times N(Z/Y\cap Z)$ and $N^*(\Xo/\Yo)\simeq\Ru{P}\times N^*(Z/Y\cap Z)$. Thus replacing $G$ with~$L$, $X$ with~$Z$, and $Y$ with~$Y\cap Z$ in Theorem~\ref{norm&conorm} preserves complexities, ranks, and weight lattices. We may now assume that $X$ is affine and $Y$ is closed in~$X$.

Let $I\vartriangleleft R=\kk[X]$ be the ideal defining~$Y$. There is a $G$-stable descending filtration of $R$ by the powers of~$I$. Let $\gr R=\bigoplus_{n=0}^{\infty}I^n/I^{n+1}$ denote the associated graded algebra. Note that
\begin{equation*}
    \gr R \simeq \Sym^{\bullet}_{R/I}(I/I^2) \simeq \kk[N]
\end{equation*}
\cite[II.8]{AG} and, by complete reducibility of $G$-modules, $\gr R$ is isomorphic to $R$ as a $G$-module.

However, for an affine variety~$X$, the complexity, rank, and weight lattice are read off the $G$-module structure of~$\kk[X]$. Namely, in the notation of Remark~\ref{loc.str.gen}, $P$~is the common stabilizer in $G$ of all $B$-stable lines in~$R$, $L_0$~is the common stabilizer in $L$ of all $B$-eigenvectors in~$R$, and $\WL{X}$ is spanned by their eigenweights \cite[II.3.6]{comm&coiso}. This observation proves Theorem~\ref{norm&conorm} for normal bundles.

As for conormal bundles, we make use of the following lemma, which may be of independent interest.

\begin{lemma}\label{c&r(E*)}
Let $E\to Y$ be a $G$-vector bundle over a $G$-variety $Y$ and $E^*$ be the dual bundle. Then $c(E^*)=c(E)$, $r(E^*)=r(E)$, and $\WL{E^*}=\WL{E}^*$, where the involution ${}^*$ is defined as in the Introduction by the longest element of the Weyl group of $T$ in~$L$, in the notation of the Local Structure Theorem for $X=Y$.
\end{lemma}

\begin{proof}
From the Local Structure Theorem (for $X=Y$) it is clear that the complexities, ranks, and weight lattices of $E$ and $E^*$ do not change if we pass from $G$ to~$L$, from $Y$ to~$Z$, and from $E$ to~$E|_Z$. Indeed, $E|_{\Yo}\simeq\Ru{P}\times E|_Z$ and $E^*|_{\Yo}\simeq\Ru{P}\times E^*|_Z$ (as $P$-varieties).

Take a general point $z\in Z$. The group $L_0$ fixes $z$ and acts on the fiber $E_z$ over~$z$. This group may be disconnected, but it is a direct product of a connected reductive group and a finite Abelian group. The highest weight theory for representations and the notions of complexity, rank, etc extend to such groups word by word, except that the ``weight lattice'' may be no longer a lattice, but a finitely generated Abelian group.

We claim that $c(E)=c(Z)+c(E_z)$, $r(E)=r(Z)+r(E_z)$, and there is an exact sequence
\begin{equation*}
    0 \longrightarrow \WL{Z} \longrightarrow \WL{E} \longrightarrow \WL{E_z} \longrightarrow 0,
\end{equation*}
where the complexity, rank, and weight group of $E_z$ are computed for the action of~$L_0$, cf.~\cite[Thm.\,9.4]{hom&emb}. (A similar assertion holds, of course, for~$E^*$.) Indeed, take a general point $v\in E_z$. The orbit $(B\cap L)v$ is fibered over $(B\cap L)z=Lz \simeq L/L_0=A$ with the fiber $(B\cap L_0)v$. Hence $c(E)=\codim_E(B\cap L)v= \codim_ZLz+\codim_{E_z}(B\cap L_0)v=c(Z)+c(E_z)$. The weight lattice $\WL{E|_Z}$ consists of the eigenweights of all $(B\cap L)$-semi-invariant functions on $(B\cap L)v$, and similarly for~$\WL{E_z}$. As $(B\cap L)/(B\cap L_0)=A$ is a torus, every $(B\cap L_0)$-semi-invariant function on $(B\cap L_0)v$ extends to a $(B\cap L)$-semi-invariant function on $(B\cap L)v$. Therefore the restriction of functions yields the above exact sequence and the equality on ranks holds.

Since $\kk[E^*_z]$ is the dual $L_0$-module of~$\kk[E_z]$, we have $c(E_z^*)=c(E_z)$, $r(E_z^*)=r(E_z)$, and $\WL{E_z^*}=\WL{E_z}^*$. In view of Remark~\ref{loc.str.gen}, the involution ${}^*$ acts on $\WL{Z}$ by inversion, which completes the proof.
\end{proof}

Applying Lemma~\ref{c&r(E*)} to $E=N$ concludes the proof of Theorem~\ref{norm&conorm}.

\begin{remark}
In the course of the proof, we have seen in fact that the stabilizers of general position for the actions of $B$ on $X$ and $N(X/Y)$ coincide and those on $N(X/Y)$ and $N^*(X/Y)$ (or, more generally, on $E$ and~$E^*$) are conjugate by a Weyl involution of~$L$. Indeed, by the Local Structure Theorem and Remark~\ref{loc.str.gen}, these stabilizers are determined by the $G$-action on~$Y$ (which determines $P$ and~$L$) and by the weight lattices (which determine the stabilizers of general positions for the actions of $B\cap L$ on $Z$ and on the vector bundles over $Y\cap Z$).
\end{remark}

\section{Lagrangian subvarieties}

\subsection{}\label{Lagr<cotg}

Recall the canonical structure of a symplectic variety on the cotangent bundle. There is a 1-form $\ell$ on $T^*X$ given by the formula $\langle\ell(p),\nu\rangle=\langle p,d\pi(\nu)\rangle$, $\forall p\in T^*X$, $\nu\in T_p(T^*X)$, where $\pi:T^*X\to X$ is the canonical projection. The 2-form $\omega=-d\ell$ is closed, nondegenerate, and thus endows $T^*X$ with a structure of a symplectic variety. In local coordinates $x_1,\dots,x_n$ on $X$ and dual coordinates $y_1,\dots,y_n$ in the cotangent spaces, $\ell=\sum y_idx_i$ and $\omega=\sum dx_i\wedge dy_i$. If $G$ acts on~$X$, then the induced $G$-action on $T^*X$ preserves the symplectic structure.

As noted above, $N^*=N^*(X/Y)$ is a Lagrangian subvariety in~$T^*X$, which is $G$-stable provided $Y$ is so. In fact, any conical Lagrangian subvariety in $T^*X$ is a conormal bundle, or at least shares a common open subset with the latter. (A subvariety in $T^*X$ is said to be conical if it is stable under the $\kk^{\times}$-action on $T^*X$ by dilatations in each cotangent space.) In \cite{conormal} Panyushev posed a question whether Theorem~\ref{norm&conorm} can be generalized to arbitrary invariant Lagrangian subvarieties in~$T^*X$. We prove such a generalization here (Corollary~\ref{Lagr<T*X}). In order to do this, we need a better understanding of the structure of Lagrangian subvarieties in cotangent bundles.

\begin{proposition}
Let $S\subset T^*X$ be a Lagrangian subvariety. There exists a smooth subvariety $Y\subset X$ and a Lagrangian subvariety $C\subset T^*Y$ covering $Y$ under the projection $T^*Y\to Y$ such that $S$ shares a common open subset with~$\rho^{-1}(C)$, where $\rho:T^*X|_Y\to T^*Y$ is the restriction map.
\end{proposition}

\begin{proof}
Put $Y=\pi(S)$. Shrinking $S$ (by passing to an open subset) we may assume that $Y$ is a smooth subvariety and $\pi:S\to Y$ is submersive. Take any $p\in S$ and put $y=\pi(p)$. We have the following exact sequences:
\begin{equation*}
\begin{array}{ccccccccc}
    0 &\longrightarrow&      T^*_yX      &\longrightarrow& T_p(T^*X) &\longrightarrow& T_yX &\longrightarrow& 0 \\[0.5ex]
      &               &     \bigcup      &               &  \bigcup  &              &\bigcup&               &   \\[0.5ex]
    0 &\longrightarrow& T_pS \cap T^*_yX &\longrightarrow&   T_pS    &\longrightarrow& T_yY &\longrightarrow& 0\rlap. \\
\end{array}
\end{equation*}
For any $\nu\in T_pS$ and $q\in T_pS\cap T^*_yX$ we have $\omega(\nu,q)=\langle d\pi(\nu),q\rangle=0$, whence $q\in N^*_y$. As $\dim S=\dim X$, we have $T_pS\cap T^*_yX=N^*_y$. Hence $S\cap T^*_yX$ is an open subset in a finite union $\bigcup_i(p_i+N^*_y)$ of translates of the fiber of the conormal bundle. This observation proves, in particular, the above claim about conical Lagrangian subvarieties.

The translates $p_i+N^*_y$ are fibers of $\rho$ and $C=\rho(S)$ is a smooth subvariety in $T^*Y$ covering~$Y$. Since the restriction of $\omega$ to $T^*X|_Y$ is the pullback of the symplectic form on~$T^*Y$ and $S$ is Lagrangian, $C$~is a Lagrangian subvariety in~$T^*Y$.
\end{proof}

\begin{corollary}\label{Lagr<T*X}
Suppose that $S\subset T^*X$ is a $G$-stable Lagrangian subvariety. Then $c(S)=c(X)$ and $r(S)=r(X)$.
\end{corollary}

\begin{proof}
Arguing as in the proof of Lemma~\ref{c&r(E*)}, we see that, in the notation of Theorem~\ref{norm&conorm}, $c(N^*)=c(Y)+c(N^*_y)$, $r(N^*)=r(Y)+r(N^*_y)$ for general $y\in Y\cap Z$ and $c(S)=c(Y)+c(p+N^*_y)$, $r(S)=r(Y)+r({p+N^*_y})$ for $p\in\rho^{-1}(C\cap T^*_yY)$, where the complexities and ranks of $N^*_y$ and $p+N^*_y$ are computed for the action of~$L_0^{\circ}$. Since $L_0^{\circ}$ is connected, it fixes $C\cap T^*_yY$ pointwise, and complete reducibility of $L_0^{\circ}$-modules implies that $N^*_y$ and $p+N^*_y$ are isomorphic $L_0^{\circ}$-varieties, whence the claim.
\end{proof}

\subsection{}

It is well-known that the $G$-action on $M=T^*X$ is Hamiltonian. This means that there exists a $G$-equivariant \emph{moment map} $\Phi:M\to\g^*$ such that for any $\xi\in\g$, considered as a linear function on~$\g^*$, the skew gradient of its pullback $\Phi^*\xi\in\kk[M]$ is the velocity field~$\xi_*$. Equivalently, $\langle d_p\Phi(\nu),\xi\rangle=\omega(\xi{p},\nu)$, $\forall p\in M$, $\nu\in T_pM$, $\xi\in\g$, where $\xi{p}$ is the velocity vector of $\xi$ at~$p$. See \cite[II.2]{comm&coiso} for more details. The moment map of the cotangent bundle is given by a formula $\langle\Phi(p),\xi\rangle=\langle p,\xi{x}\rangle$, $\forall x\in X$, $p\in T^*_xX$, $\xi\in\g$.

Is it possible to generalize Theorem~\ref{norm&conorm} and Corollary~\ref{Lagr<T*X} to arbitrary Hamiltonian symplectic varieties instead of cotangent bundles? Even a formulation of such a generalization is not obvious. Indeed, for arbitrary $M$ we have no base variety~$X$, and it is not clear what are the substitutes for $c(X)$ and~$r(X)$. Luckily, it follows from the results of Knop \cite{Weyl&mom} and Vinberg \cite{comm&coiso} that  $c(X)$ and $r(X)$ are in fact symplectic invariants of the Hamiltonian action on~$T^*X$.

Let $M$ be any Hamiltonian $G$-variety.

\begin{definition}
The \emph{corank} of $M$ is $\cork M=\rk\omega|_{(\g{p})^{\sort}}$ and the \emph{defect} of $M$ is $\df M= \dim\Ker\omega|_{(\g{p})^{\sort}}=\dim\g{p}\cap(\g{p})^{\sort}$, where $p\in M$ is a general point. Here $\g{p}=T_pGp$ is the tangent space of an orbit (= the set of all velocity vectors) and ${}^{\sort}$ denotes the skew-orthocomplement.
\end{definition}

It is clear from the above that $\Ker d_p\Phi=(\g{p})^{\sort}$ and $\Im d_p\Phi=(\g_p)^{\ann}$ is the annihilator in $\g^*$ of the isotropy subalgebra at~$p$. It follows that $\dim\overline{\Phi(M)}=\dim Gp$ for general $p\in M$, whence
\begin{gather*}
    \df M   = \dim Gp - \dim G\Phi(p) = \dim\overline{\Phi(M)}/G, \\
    \cork M = \dim M/G - \df M        = \dim M - \dim\overline{\Phi(M)} - \dim\overline{\Phi(M)}/G.
\end{gather*}
(Here the quotient by $G$ means the rational quotient, i.e., the quotient space of an invariant open subset for which a geometric quotient exists.)

\begin{theorem}[{\cite[7.1]{Weyl&mom}, \cite[II.3.4]{comm&coiso}, see also \cite[Thm.\,8.17]{hom&emb}}] \label{cork&def(T*X)}
\begin{equation*}
2c(X)=\cork T^*X,\qquad r(X)=\df T^*X.
\end{equation*}
\end{theorem}

Now we can formulate our generalization of Corollary~\ref{Lagr<T*X}.

\begin{theorem}\label{c&r(Lagr)}
Let $M$ be a Hamiltonian $G$-variety and $S\subset M$ be a $G$-stable Lagrangian subvariety. Then $2c(S)=\cork M$ and $r(S)=\df M$.
\end{theorem}

In fact, we shall prove a more precise statement. Note that the moment map $\Phi:M\to\g^*$ is defined uniquely up to a shift by a $G$-fixed vector in~$\g^*$. Since the subvariety $S\subset M$ is isotropic and $G$-stable, $d\Phi$~vanishes on $TS\subset TM$, whence $\Phi(S)$ is a $G$-fixed vector. Shifting by the opposite vector, we may assume that $\Phi(S)=0$ and thus define the moment map uniquely. In view of Theorem~\ref{cork&def(T*X)} and the above formul{\ae} for the corank and defect, Theorem~\ref{c&r(Lagr)} stems from the following result:

\begin{theorem}\label{mom(Lagr)}
The closures of the images of the moment maps for $M$ and $T^*S$ coincide.
\end{theorem}

\subsection{}\label{deform}

The main idea of the proof is deformation (or, strictly speaking, contraction) to the normal bundle \cite[5.1]{int}.

We may assume $S$ to be closed in~$M$. Consider the product $M\times\AAA^1$ of $M$ with the coordinate affine line and blow up $S\times\{0\}\subset M\times\AAA^1$. The exceptional divisor is isomorphic to the projective bundle $\PP(N\oplus\kk)$ over $S\times\{0\}$, where $N$ is the normal bundle of $S\subset M$. The strict preimage $\check{M}$ of $M\times\{0\}$ is nothing else but the blowup of $M\times\{0\}$ at $S\times\{0\}$.  These two divisors intersect in~$\PP(N)$, the exceptional divisor of $\check{M}\to M$. Removing $\check{M}$ we obtain a smooth variety $\widehat{M}$ together with a smooth morphism $\delta:\widehat{M}\to\AAA^1$ such that $\delta^{-1}(\AAA^1\setminus\{0\})\simeq M\times(\AAA^1\setminus\{0\})$ and $\delta^{-1}(0)\simeq N$. Furthermore, $\delta$~is equivariant with respect to the $\kk^{\times}$-actions on $\AAA^1$ by dilatations and on $\widehat{M}$ coming from the action on $M\times\AAA^1$ by dilatations of the second factor. The $\kk^{\times}$-action on $\delta^{-1}(0)$ is nothing else but the action on the vector bundle $N$ by inverse dilatations in the fibers.

In more algebraic terms, $\varphi:\widehat{M}\to M\times\AAA^1\to M$ is an affine morphism and $\varphi_*\Oo_{\widehat{M}}=\bigoplus_{n=-\infty}^{\infty}\Ii_S^nt^{-n}\subset\Oo_M[t^{\pm1}]$, where $t$ is the coordinate on~$\AAA^1$, $\Ii_S\lhd\Oo_M$ is the ideal sheaf defining~$S$, and $\Ii_S^{-1}=\Ii_S^{-2}=\dots=\Oo_M$ by definition.

\begin{lemma}\label{N=T*S}
$N\simeq T^*S$.
\end{lemma}

\begin{proof}
For any $p\in S$, $T_pS$ is a Lagrangian subspace in~$T_pM$, whence the symplectic form $\omega$ induces a nondegenerate pairing between $T_pS$ and $T_pM/T_pS=N_p$.
\end{proof}

Recall that the Poisson bracket on $\Oo_M$ is defined by the formula $\{f,g\}=\omega(\nabla{f},\nabla{g})$, where $\nabla{f}$ is the skew gradient of a function~$f$, which satisfies the condition $\omega(\nabla{f},\cdot)=df$. The Poisson bracket endows $\Oo_M$ with a structure of a sheaf of Poisson algebras, i.e., it is a Lie bracket satisfying the Leibniz identity
\begin{equation*}
    \{f,gh\} = \{f,g\}\cdot h + g\cdot\{f,h\}, \qquad \forall f,g,h\in\Oo_M.
\end{equation*}
The 1-st order differential operator $\{f,\cdot\}$ is nothing else, but the Lie derivative along~$\nabla{f}$. See \cite[II.1.2--3]{comm&coiso} for details.

The Poisson structure on a cotangent bundle can be described as follows \cite[II.1.4]{comm&coiso}. In the notation of~\ref{Lagr<cotg}, $\pi_*\Oo_{T^*X}$~is generated by the functions $\pi^*f$, where $f$ is a function on (an open subset of)~$X$, and the vector fields $\xi$ on $X$, regarded as fiberwise linear functions on~$T^*X$. The respective Poisson brackets are: $\{\pi^*f,\pi^*g\}=0$, $\{\xi,\pi^*f\}=\pi^*(\xi{f})$, $\{\xi,\eta\}=[\xi,\eta]$. (Here $\xi{f}$ is the Lie derivative of $f$ along~$\xi$ and $[\xi,\eta]$ is the commutator of vector fields defined in such a way that the respective Lie derivative operator is the commutator of the Lie derivatives corresponding to $\xi,\eta$.)

There is yet another description of the Poisson structure on~$T^*X$. Namely the noncommutative algebra sheaf of differential operators on $X$ has an increasing filtration $\Dd_X=\bigcup_{n=0}^{\infty}\Dd_X^{(n)}$ by the order of a differential operator. The associated graded algebra sheaf $\gr\Dd_X$ is commutative and in fact $\gr\Dd_X\simeq\pi_*\Oo_{T^*X}$. The map $\Dd_X^{(n)}\to\Dd_X^{(n)}/\Dd_X^{(n-1)}\simeq\Sym_{\Oo_X}^n(\Tt_X)$, where $\Tt_X$ is the sheaf of vector fields, is known as the \emph{symbol map}. For any $\partial\in\Dd_X^{(n)}$, $\partial'\in\Dd_X^{(m)}$, we have  $[\partial,\partial']\in\Dd_X^{(n+m-1)}$, and it is easy to deduce from the previous paragraph that
\begin{equation*}
\{\partial\bmod\Dd_X^{(n-1)},\partial'\bmod\Dd_X^{(m-1)}\}=[\partial,\partial']\bmod\Dd_X^{(n+m-2)}.
\end{equation*}

We can equip $\widehat{M}$ with a natural Poisson structure. Since the subvariety $S\subset M$ is coisotropic (which means that $TS\supset(TS)^{\sort}$ in~$TM|_S$), the skew gradients of $f\in\Ii_S$ are tangent to~$S$, i.e., $\langle d\Ii_S,\nabla\Ii_S\rangle=0$ on~$S$ or, equivalently, $\{\Ii_S,\Ii_S\}\subset\Ii_S$. It easily follows that $\{\Ii_S^n,\Ii_S^m\}\subset\Ii_S^{n+m-1}$, $\forall n,m\in\ZZ$. Now the Poisson bracket on $\varphi_*\Oo_{\widehat{M}}$ is defined as $\{ft^{-n},gt^{-m}\}=\{f,g\}t^{-n-m+1}$, $\forall f\in\Ii_S^n$, $g\in\Ii_S^m$.

The Poisson variety $\widehat{M}$ is not symplectic and the Poisson bracket of functions can be computed fiberwise along symplectic leaves $M_c=\delta^{-1}(c)$, $c\in\AAA^1$. Let $\omega_c$ and $\{\cdot,\cdot\}_c$ denote the symplectic form and Poisson bracket on~$M_c$, respectively.

\begin{lemma}\label{def.sympl}
If $c\ne0$, then $\omega_c=\omega/c$ and $\{\cdot,\cdot\}_c=c\{\cdot,\cdot\}$ on $M_c\simeq M$, while $\omega_0$ and $\{\cdot,\cdot\}_0$ are the standard symplectic form and Poisson bracket on $M_0\simeq T^*S$.
\end{lemma}

\begin{proof}
The assertion is obvious for $c\ne0$. For $c=0$ we observe that $\varphi_*\Oo_{M_0}=\bigoplus_{n=0}^{\infty}\Ii_S^n/\Ii_S^{n+1}\cdot t^{-n}$ is generated by $\Oo_S=\Oo_M/\Ii_S$ and~$\Nn^*t^{-1}$, where $\Nn^*=\Ii_S/\Ii_S^2$ is the conormal sheaf of~$S$, so it suffices to compute the Poisson bracket of generators. For $f,g\in\Oo_M$ we have
\begin{equation*}
\{f\bmod\Ii_S,g\bmod\Ii_S\}=(\{f,g\}\bmod\Ii_S^0)t=0.
\end{equation*}
If $f\in\Ii_S$, then
\begin{equation*}
\{(f\bmod\Ii_S^2)t^{-1},g\bmod\Ii_S\}=\{f,g\}\bmod\Ii_S=(\nabla{f})g|_S.
\end{equation*}
If both $f,g\in\Ii_S$, then
\begin{equation*}
\{(f\bmod\Ii_S^2)t^{-1},(g\bmod\Ii_S^2)t^{-1}\}=(\{f,g\}\bmod\Ii_S^2)t^{-1}
\end{equation*}
and $\{f,g\}\bmod\Ii_S^2$ is identified with $\nabla\{f,g\}=[\nabla{f},\nabla{g}]$ restricted to $S$ via $\Nn^*\simeq\Tt_S$. Since $S$ is Lagrangian, $\Tt_S$~is generated by the vector fields~$\nabla{f}$, $f\in\Ii_S$, and, taking into account the above description of the Poisson bracket on a cotangent bundle, we conclude the proof.
\end{proof}

The moment map of $M$ can be deformed as well. Namely consider the \emph{total moment map} $\widehat\Phi:\widehat{M}\to\g^*\times\AAA^1$ such that the dual algebra homomorphism $\widehat\Phi^*:\kk[\g^*][t]\to\kk[\widehat{M}]$ is defined by the formul\ae: $\widehat\Phi^*\xi=\Phi^*\xi\cdot t^{-1}$, $\forall\xi\in\g$, and $\widehat\Phi^*t=t$. Clearly, $\widehat\Phi$ maps $M_c$ to $\g^*\times\{c\}\simeq\g^*$.

\begin{lemma}
$\Phi_c=\widehat\Phi|_{M_c}$ is the moment map for the Hamiltonian $G$-action on $M_c$ equipped with the symplectic structure as in Lemma~\ref{def.sympl}.
\end{lemma}

\begin{proof}
For $c\ne0$ we have $\widehat\Phi(p,c)=(\Phi(p)/c,c)$ on $M_c\simeq M\times\{c\}$, whence the claim. For $c=0$ we have $\Phi_0^*\xi=(\Phi^*\xi\bmod\Ii_S^2)t^{-1}$, $\forall\xi\in\g$, and $\Phi^*\xi\bmod\Ii_S^2$ is identified with~$\nabla(\Phi^*\xi)|_S$, which is the velocity field of $\xi$ on $S$ by the definition of the moment map. But this velocity field, considered as a fiberwise linear function on~$T^*S$, has the velocity field of $\xi$ on $T^*S$ as its skew gradient \cite[II.1.4, 2.1]{comm&coiso}.
\end{proof}

Thus the Hamiltonian structure on $M$ is contracted to the Hamiltonian structure on~$T^*S$.

\begin{remark}
In differential geometry, for a Hamiltonian action of a compact Lie group $G$ on a symplectic manifold $M$ with a $G$-stable Lagrangian submanifold~$S$, it follows from the equivariant Darboux--Weinstein theorem \cite[\S22]{symp} that $M$ is locally isomorphic (as a Hamiltonian manifold) to $T^*S$ in a neighborhood of~$S$, cf.~\cite[Chap.\,IV, Prop.\,1.1]{geom.as}. For Hamiltonian actions of reductive algebraic groups this is no longer true. To construct a counterexample, it suffices to find a point $p\in S$ such that $T_pS$ has no $G_p^{\circ}$-stable complement in~$T_pM$. (This could not happen if $M$ were locally isomorphic to $T^*S$ in a neighborhood of~$p$, even in {\'e}tale topology.)

For instance, consider the variety $X$ of complete conics, which is obtained from the space $\PP^5=\PP(\Sym^2\kk^3)^*$ of plane conics by blowing up the surface of double lines, see e.g.\ \cite[Ex.\,17.12]{hom&emb}. One may define $X$ by considering the dual projective space $(\PP^5)^*=\PP(\Sym^2\kk^3)$ and taking the subvariety in $\PP^5\times(\PP^5)^*$ given by the equation ``$x\cdot x'$ is a scalar matrix'', where $x,x'$ are symmetric $3\times3$ matrices of homogeneous coordinates in $\PP^5$ and~$(\PP^5)^*$. If $x$ is the matrix of a nondegenerate quadratic form representing a smooth conic in~$\PP^2$, then $x'$ represents the dual conic in~$(\PP^2)^*$. The group $G=\SL_3(\kk)$ acts on $X$ in a natural way with four orbits distinguished by the pair of values $(\rk{x},\rk{x'})=(3,3)$, $(2,1)$, $(1,2)$, or~$(1,1)$. Take the closed orbit $Y=\{\rk{x}=\rk{x'}=1\}\subset X$ and put $M=T^*X$, $S=N^*(X/Y)$. If $y\in Y$ is the pair of $B$-stable double lines in $\PP^2$ and~$(\PP^2)^*$, then $G_y=B$ acts on the conormal space $N^*_y$ by two linearly independent weights (doubled simple roots). Hence a general point $p\in N^*_y$ has open orbit in $S$ and $G_p^{\circ}=U$. Explicit calculations show that $T_pS$ has no $U$-stable complement in~$T_pM$.
\end{remark}

\subsection{}\label{deg.horosph}

In order to study the image of the moment map, we need some results from the forthcoming paper~\cite{cotangent}. For convenience of the reader, we reproduce them here.

Let $Y$ be a smooth $G$-variety. We apply the Local Structure Theorem to $X=Y$ and use the notation therefrom.

Consider any faithful linear representation $G\hookrightarrow\GL_n(\kk)$ and the respective $G$-invariant inner product $(\xi,\eta)=\tr(\xi\eta)$ on~$\g$. This allows to identify $\g$ with~$\g^*$. The inner product is nondegenerate on $\lv$ and~$\lv_0$, since both groups $L,L_0$ are reductive. The orthocomplement $\ab$ of $\lv_0$ in $\lv$ is a toric subalgebra identified with the Lie algebra of~$A$. Put $M=Z_G(\ab)$. The subalgebras $\Ru\p,\Ru[-]\p$ are isotropic subspaces orthogonal to~$\lv$, and the inner product puts them in duality.

Choose a 1-parameter subgroup $\gamma:\kk^{\times}\to T\cap L_0$ defining the parabolic $M\cap P^-\subset M$, i.e., the eigenweights of $\gamma$ on $\m\cap\Ru[-]\p$ are positive and on $\lv$ are zero. Similarly, $\gamma$~defines a parabolic subgroup $\overline{Q}\subset G$, with the Levi decomposition $\overline{Q}=\Ru{\overline{Q}}\leftthreetimes\overline{M}$. Note that
\begin{align*}
    \overline{Q}      &= \{g\in G\mid\exists\lim_{t\to0}\gamma(t)g\gamma(t)^{-1}\}, \\
    \Ru{\overline{Q}} &= \{g\in G\mid\lim_{t\to0}\gamma(t)g\gamma(t)^{-1}=e\},      \\
    \overline{M}      &= \{g\in G\mid\gamma(t)g\gamma(t)^{-1}=g\},
\end{align*}
and $\overline{M}\cap M=L$. Then $\overline{U}=\Ru{\overline{Q}}\leftthreetimes(\overline{M}\cap U)$ is a maximal unipotent subgroup of $G$ normalized by~$T$, and $\overline{P}=\Ru{\overline{Q}}\leftthreetimes(\overline{M}\cap P)$ is a parabolic sharing the Levi subgroup $L$ with~$P$. We represent the root systems of various parabolics and their Levi subgroups at Figure~\ref{parab}.

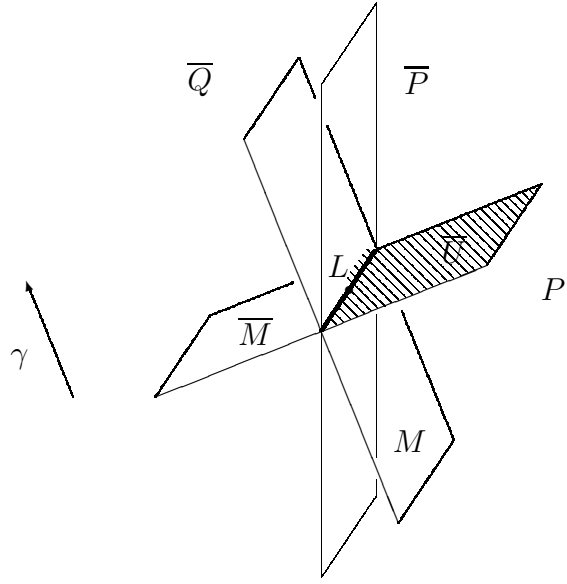
\begin{figure}
\unitlength .4ex 
\linethickness{0.4pt}
\ifx\plotpoint\undefined\newsavebox{\plotpoint}\fi 
\begin{picture}(97,105)(0,0)
\multiput(25,33)(.11904762,.17857143){84}{\line(0,1){.17857143}}
\multiput(85,57)(.11904762,.17857143){84}{\line(0,1){.17857143}}
\multiput(69,10)(.11904762,.17857143){84}{\line(0,1){.17857143}}
\multiput(41,80)(.11904762,.17857143){84}{\line(0,1){.17857143}}
\multiput(65,60)(.3,.12){100}{\line(1,0){.3}}
\multiput(10,33)(-.11940299,.29850746){67}{\line(0,1){.29850746}}
\put(2,53){\vector(-1,3){.417}}
\multiput(51,95)(.12,-.3){25}{\line(0,-1){.3}}
\multiput(65,60)(-.12,.3){75}{\line(0,1){.3}}
\multiput(79,25)(-.11875,.296875){80}{\line(0,1){.296875}}
\multiput(35,48)(.3,.12){50}{\line(1,0){.3}}
\put(2,40){\makebox(0,0)[rc]{$\gamma$}}
\put(43,45){\makebox(0,0)[cc]{$\overline{M}$}}
\put(71,25){\makebox(0,0)[cc]{$M$}}
\put(60,55){\makebox(0,0)[rb]{$L$}}
\put(79,60){\makebox(0,0)[cc]{$\overline{U}$}}
\put(70,91){\makebox(0,0)[lc]{$\overline{P}$}}
\put(97,52.75){\makebox(0,0)[cc]{$P$}}
\put(33,91){\makebox(0,0)[cc]{$\overline{Q}$}}
\put(25,33){\line(5,2){60}}
\put(41,80){\line(2,-5){28}}
\linethickness{1.5pt}
\multiput(55,45)(.11904762,.17857143){84}{\line(0,1){.17857143}}
\linethickness{0.05pt}
\put(55,90){\line(0,-1){90}}
\put(65,105){\line(0,-1){45}}
\put(65,15){\line(0,1){3}}
\put(65,22){\line(0,1){26}}
\put(60,52.5){\circle*{1.5}}
\multiput(55,90)(.0299401198,.0449101796){334}{\line(0,1){.0449101796}}
\multiput(55,0)(.0299401198,.0449101796){334}{\line(0,1){.0449101796}}
\multiput(65,60)(.0299227799,-.0299227799){259}{\line(0,-1){.0299227799}}
\multiput(66.25,60.5)(.0299227799,-.0299227799){259}{\line(0,-1){.0299227799}}
\multiput(67.5,61)(.0299227799,-.0299227799){259}{\line(0,-1){.0299227799}}
\multiput(68.75,61.5)(.0299227799,-.0299227799){259}{\line(0,-1){.0299227799}}
\multiput(70,62)(.0299227799,-.0299227799){259}{\line(0,-1){.0299227799}}
\multiput(71.25,62.5)(.0299227799,-.0299227799){259}{\line(0,-1){.0299227799}}
\multiput(72.5,63)(.0299227799,-.0299227799){259}{\line(0,-1){.0299227799}}
\multiput(73.75,63.5)(.0299227799,-.0299227799){259}{\line(0,-1){.0299227799}}
\multiput(75,64)(.0299227799,-.0299227799){259}{\line(0,-1){.0299227799}}
\multiput(76.25,64.5)(.0299227799,-.0299227799){259}{\line(0,-1){.0299227799}}
\multiput(77.5,65)(.0299227799,-.0299227799){259}{\line(0,-1){.0299227799}}
\multiput(78.75,65.5)(.02991453,-.02991453){234}{\line(0,-1){.02991453}}
\multiput(80,66)(.03,-.03){220}{\line(0,-1){.03}}
\multiput(81.25,66.5)(.029904306,-.029904306){203}{\line(0,-1){.029904306}}
\multiput(82.5,67)(.029891304,-.029891304){184}{\line(0,-1){.029891304}}
\multiput(83.75,67.5)(.02994012,-.02994012){167}{\line(0,-1){.02994012}}
\multiput(85,68)(.03,-.03){148}{\line(0,-1){.03}}
\multiput(86.25,68.5)(.029850746,-.029850746){130}{\line(0,-1){.029850746}}
\multiput(87.5,69)(.02991453,-.02991453){110}{\line(0,-1){.02991453}}
\multiput(88.75,69.5)(.029891304,-.029891304){92}{\line(0,-1){.029891304}}
\multiput(90,70)(.03,-.03){75}{\line(0,-1){.03}}
\multiput(91.25,70.5)(.02966102,-.02966102){55}{\line(0,-1){.02966102}}
\multiput(92.5,71)(.0297619,-.0297619){40}{\line(0,-1){.0297619}}
\multiput(71.5,51.75)(-.03,.03){275}{\line(0,1){.03}}
\multiput(70.25,51.25)(-.0299227799,.0299227799){259}{\line(0,1){.0299227799}}
\multiput(69,50.75)(-.029958678,.029958678){242}{\line(0,1){.029958678}}
\multiput(67.75,50.25)(-.03,.03){225}{\line(0,1){.03}}
\multiput(66.5,49.625)(-.029904306,.029904306){209}{\line(0,1){.029904306}}
\multiput(65.25,49.25)(-.029891304,.029891304){184}{\line(0,1){.029891304}}
\multiput(64,48.75)(-.03,.03){125}{\line(0,1){.03}}
\multiput(62.75,48.25)(-.029816514,.029816514){109}{\line(0,1){.029816514}}
\multiput(61.5,47.75)(-.029891304,.029891304){92}{\line(0,1){.029891304}}
\multiput(60.25,47.25)(-.03,.03){75}{\line(0,1){.03}}
\multiput(59,46.75)(-.02966102,.02966102){59}{\line(0,1){.02966102}}
\multiput(57.75,46.25)(-.0297619,.0297619){42}{\line(0,1){.0297619}}
\end{picture}
  \caption{Roots of various subgroups}\label{parab}
\end{figure}

\begin{proposition}\strut\newline
\begin{enumerate}
  \item $\Ru{\overline{Q}}$-orbits of the points of $Z$ coincide with their $(\Ru{\overline{Q}}\cap\Ru{P})$-orbits.
  \item $\overline{U}$-orbits of the points of $Z$ coincide with their $(\overline{U}\cap\Ru{P})$-orbits.
  \item $\overline{U}Z\simeq(\Ru{\overline{Q}}\cap\Ru{P})\times(\overline{M}\cap\Ru{P})\times Z$ is the set of points $y\in Y$ such that $\lim_{t\to0}\gamma(t)y$ exists in~$\Yo$, and taking this limit is a retraction onto the set $(\overline{M}\cap\Ru{P})\times Z$ of $\gamma$-fixed points in~$\Yo$.
\end{enumerate}
\end{proposition}

\begin{proof}
By \cite[28.1]{LAG}, the group multiplication induces a $T$-equivariant isomorphism of varieties $\Ru{P}\simeq(\Ru[-]{\overline{Q}}\cap\Ru{P})\times(\Ru{\overline{Q}}\cap\Ru{P})\times(\overline{M}\cap\Ru{P})$ (where $T$ acts on subgroups by conjugation), whence
\begin{equation*}
\Yo\simeq(\Ru[-]{\overline{Q}}\cap\Ru{P})\times(\Ru{\overline{Q}}\cap\Ru{P})\times(\overline{M}\cap\Ru{P})\times Z.
\end{equation*}
The subset of $\gamma$-fixed points is $\Yo^{\gamma}\simeq(\overline{M}\cap\Ru{P})\times Z$. The limit $y_0=\lim_{t\to0}\gamma(t)y$ is a $\gamma$-fixed point. If $y_0\in\Yo$, then $\gamma(t)y\in\Yo$ for $t$ in a neighborhood of~$0$, but $\Yo$ is $\gamma$-stable, whence $\gamma(t)y\in\Yo$ for all~$t$. Hence $y=u_-u_+u_0z$, where $z\in Z$, $u_{\pm}\in\Ru[\pm]{\overline{Q}}\cap\Ru{P}$, $u_0\in\overline{M}\cap\Ru{P}$. Applying $\gamma(t)$ to this equality and taking into account that $\nexists\lim_{t\to0}\gamma(t)u_-\gamma(t)^{-1}$ (unless $u_-=e$), $\lim_{t\to0}\gamma(t)u_+\gamma(t)^{-1}=e$, and $\gamma(t)$ commutes with $u_0$ and fixes~$z$, we see that $\lim_{t\to0}\gamma(t)y$ exists in $\Yo$ if and only if $u_-=e$, and then $y_0=u_0z$.

For $q\in\Ru{\overline{Q}}$ we have $\gamma(t)qz=\gamma(t)q\gamma(t)^{-1}z\to z$ as $t\to0$. By the above, $qz\in\Yo$ and $qz=u_+z$ for some $u_+\in\Ru{\overline{Q}}\cap\Ru{P}$. Also, $\overline{U}\cap L=U\cap L_0$ fixes~$z$. This completes the proof.
\end{proof}

Let
\begin{equation*}
\Uu = \{ p\in T^*Y \mid y=\pi(p)\in\overline{U}Z,\ \langle p,\overline{\un}y\rangle=0 \}
\end{equation*}
be the conormal bundle of the foliation of $\overline{U}$-orbits in $\overline{U}Z$. The subgroup $\overline{P}_0=\Ru{\overline{P}}\leftthreetimes L_0$ leaves each of these $\overline{U}$-orbits stable, and $\overline{P}$ permutes them, hence preserves~$\Uu$. Similarly, $P_0=\Ru{P}\leftthreetimes L_0$ preserves general $U$-orbits in~$Y$ and $P$ permutes them.

\begin{theorem}\label{horosph}
$G\Uu$ is dense in~$T^*Y$, $\Phi(\Uu)$ is dense in~$\overline\p_0^{\ann}=\ab\oplus\Ru{\overline\p}$, and $\overline{\Im\Phi}=G\overline\p_0^{\ann}=G\p_0^{\ann}$.
\end{theorem}

\begin{proof}
As elements of $\Uu$ vanish on tangent spaces of $\overline{P}_0$-orbits, we have $\Phi(\Uu)\subset\overline\p_0^{\ann}$. Consider a fiber $\Uu_z$, $z\in Z$. The subspace $\Phi(\Uu_z)$ projects onto $\ab\oplus(\Ru{\overline\p}\cap\Ru[-]\p)$ along $\Ru{\overline\p}\cap\Ru\p$. Indeed, $\ab\oplus(\Ru[-]{\overline\p}\cap\Ru\p)$ maps isomorphically onto a subspace of velocity fields $(\ab\oplus(\Ru[-]{\overline\p}\cap\Ru\p))z\subset T_zY$, which is transversal to~$\overline\un{z}$, because $\gamma$ has negative eigenweights on $(\Ru[-]{\overline\p}\cap\Ru\p)z$ and nonnegative eigenweights on~$\overline\un{z}$. Hence the inner products with elements of $\Phi(\Uu_z)$ span the dual space $(\ab\oplus(\Ru[-]{\overline\p}\cap\Ru\p))^*\simeq\ab\oplus(\Ru{\overline\p}\cap\Ru[-]\p)$.

Then $\Phi(\Uu_z)\ni\zeta=\zeta_0+\zeta_+$, where $\zeta_0\in\ab$ is a general point (by ``general'' we mean ``such that $\z_{\g}(\zeta_0)=\m$'') and $\zeta_+\in\Ru{\overline\p}\cap\Ru\p$. But $\m\cap\Ru{\overline\p}\cap\Ru\p=0$, whence $[\Ru{\overline\p}\cap\Ru\p,\zeta_0]=\Ru{\overline\p}\cap\Ru\p$, i.e., the $(\Ru{\overline{P}}\cap\Ru{P})$-orbit of $\zeta_0$ is dense in $\zeta_0+\Ru{\overline\p}\cap\Ru\p$. Since orbits of unipotent groups in affine varieties are closed \cite[1.3]{inv}, we have $(\Ru{\overline{P}}\cap\Ru{P})\zeta_0=\zeta_0+\Ru{\overline\p}\cap\Ru\p\ni\zeta$ and $[\Ru{\overline\p}\cap\Ru\p,\zeta]=\Ru{\overline\p}\cap\Ru\p$, whence $[\overline\un,\zeta]+\Phi(\Uu_z)=\overline\p_0^{\ann}$. Therefore $\Phi(\Uu)$ is dense in~$\overline\p_0^{\ann}$.

By \cite[5.5]{Ad}, $\Ru[-]{\overline{P}}$ acts on general points of $\overline\p_0^{\ann}$ with trivial stabilizer and the orbits are transversal to~$\overline\p_0^{\ann}$. Hence the same holds for general points of~$\Uu$. But $\codim\Uu=\dim Y -\dim Z=\dim\Ru{P}=\dim\Ru[-]{\overline{P}}$, whence $\Ru[-]{\overline{P}}\Uu$ is dense in~$T^*Y$. Thus $\overline{G\Uu}=T^*Y$ and $\overline{\Im\Phi}=G\overline\p_0^{\ann}$. (The latter set is closed, because $\overline\p_0^{\ann}$ is stable under a parabolic $\overline{P}\subset G$.) Again by \cite[5.5]{Ad}, $G\overline\p_0^{\ann}$ depends only on the Levi subgroup $L_0$ of~$\overline{P}_0$, whence $G\overline\p_0^{\ann}=G\p_0^{\ann}$.
\end{proof}

\begin{remark}
The equality $\overline{\Im\Phi}=G\p_0^{\ann}$ is due to Knop, with a more complicated proof \cite[5.4]{Weyl&mom}, see also \cite[8.5]{hom&emb}. For quasiaffine $Y$ one has $M=L$ \cite[II.3.6]{comm&coiso}, \cite[Prop.\,8.14]{hom&emb}, and can choose for $\gamma$ the trivial subgroup. Then $\overline{Q}=\overline{M}=G$, $\overline{U}=U$, and $\Uu$ is the conormal bundle of a foliation of general $U$-orbits in~$Y$. In this situation Theorem~\ref{horosph} was obtained by Knop \cite[3.2, 3.3]{inv.mot}, see also \cite[23.2]{hom&emb}.
\end{remark}

\subsection{}\label{proof}

In order to describe the image of the moment map for $M$ in the same way as for~$T^*S$, we need a substitute for~$\Uu$.

Let us apply the Local Structure Theorem to $X=M$, $Y=S$, and use the respective notation from \ref{LST&N*} and~\ref{deg.horosph}. Choose $m=\dim Z\cap S$ algebraically independent functions $f_1,\dots,f_m\in\kk[{Z\cap S}]$ and extend them to $L_0$-invariant functions $F_1,\dots,F_m\in\kk[Z]$. These functions are extended to $P_0$-invariant functions on $\So$ and~$\Mo$, respectively, denoted by the same letters. Shrinking~$\Mo$, we may assume that $df_i$ span the conormal spaces of the $U$-orbits in~$\So$, whence the skew gradients $\nabla{F_i}$ are linearly independent on $\So$ modulo~$T\So$.

We may interpret $F_i\in\Oo_M\subset\varphi_*\Oo_{\widehat{M}}$ as functions $\widehat{F}_i$ on (an open subset $\widehat{\Mo}$ of) $\widehat{M}$ defined by the formula
\begin{equation*}
\widehat{F}_i=
\begin{cases}
    F_i      & \text{on } M_c\simeq M, \qquad c\ne0, \\
    \pi^*f_i & \text{on } M_0\simeq T^*S.
\end{cases}
\end{equation*}
The respective Hamiltonian vector fields $\nabla\widehat{F}_i$ on $\widehat{M}$ are described fiberwise as follows:
\begin{equation*}
\nabla\widehat{F}_i|_{M_c}=
\begin{cases}
    c\nabla{F_i}, & c\ne0, \\
    -df_i,        & c=0,
\end{cases}
\end{equation*}
where $-df_i$ are regarded as vertical vector fields on $T^*S$ constant on fibers, cf.~\cite[II.1.4]{comm&coiso}. We deduce that $\nabla\widehat{F}_i$ are linearly independent on $\So\times\AAA^1\hookrightarrow\widehat{M}$ modulo~$T(\So\times\AAA^1)$.

Consider a closed subscheme $\Zz=\Phi^{-1}(\p_0^{\ann})\cap\Mo^{\gamma}\subset\Mo$. (The preimage and the intersection are meant in the schematic sense.) Note that $\Mo^{\gamma}$ is a smooth closed subset of $\Mo$ and $\nabla{F}_i$ are tangent to~$\Mo^{\gamma}$, because the functions $F_i$ and their skew gradients are $\gamma$-invariant (see Theorem~\ref{BB-dec} below). Also $\nabla{F}_i$ preserve the ideal sheaf of $\Phi^{-1}(\p_0^{\ann})$, because the latter is generated by the functions $\Phi^*\xi$, $\xi\in\p_0$, and $(\nabla{F}_i)\Phi^*\xi=\{F_i,\Phi^*\xi\}=-\xi_*F_i=0$. Hence $\nabla{F}_i$ preserve $\Ii_{\Zz}$ and induce derivations of~$\Oo_{\Zz}$.

The subscheme $\Zz$ can be put in a family $\widehat\Zz={\widehat\Phi^{-1}(\p_0^{\ann}\times\AAA^1)\cap\Bigl(\widehat\Mo\Bigr)^{\gamma}}$ of closed subschemes $\Zz_c=\widehat\Zz\cap\Mo_c$ in~$\Mo_c$, so that $\Zz=\Zz_1$ and $\Zz_0\simeq{(\overline{M}\cap\Ru{P})}\times T^*(Z\cap S)$ is the conormal bundle $\Phi_0^{-1}(\p_0^{\ann})\cap T^*\So$ of the foliation of $U$-orbits in $\So$ restricted to $\So^{\gamma}\simeq(\overline{M}\cap\Ru{P})\times(Z\cap S)$. The latter assertion follows from the $T$-equivariant isomorphism $\Phi_0^{-1}(\p_0^{\ann})\cap T^*\So\simeq\Ru{P}\times T^*(Z\cap S)$. As above, $\nabla\widehat{F}_i$ preserve $\Ii_{\widehat\Zz}$ and induce derivations of~$\Oo_{\widehat\Zz}$.

We make use of the following lemmata.

\begin{lemma}\label{sec.cone}
Let $X$ be a $\kk$-scheme of finite type and $Y\subset X$ be a closed subscheme. Every vector field on~$X$, i.e., a derivation of~$\Oo_X$ gives rise to a section of the normal cone~$C(X/Y)$.
\end{lemma}

\begin{remark}
The lemma is geometrically obvious if $X$ and $Y$ are smooth varieties. Indeed, a vector field on $X$ is a section of the tangent bundle~$TX$, which can be restricted to $Y$ and projected to the normal bundle $N(X/Y)=C(X/Y)$. But we need this assertion in a more general situation.
\end{remark}

\begin{proof}
Recall that $C(X/Y)=\Spec_{\Oo_Y}\bigoplus_{n=0}^{\infty}\Ii_Y^n/\Ii_Y^{n+1}$. We denote by $\pi:C(X/Y)\to Y$ the canonical projection. Let $\partial$ be any derivation of~$\Oo_X$. Define an $\Oo_Y$-algebra homomorphism $\check\partial:\pi_*\Oo_{C(X/Y)}\to\Oo_Y=\Oo_X/\Ii_Y$ as follows: $\check\partial=\partial^n/n!$ on $\Ii_Y^n/\Ii_Y^{n+1}$. It is easy to see that
\begin{equation*}
\check\partial(f_1\cdots f_n\bmod\Ii_Y^{n+1})=\partial{f_1}\cdots\partial{f_n}\bmod\Ii_Y
\end{equation*}
for any $f_1,\dots,f_n\in\Ii_Y$, whence $\check\partial$ is indeed a well-defined homomorphism. The dual morphism $Y\to C(X/Y)$ is right inverse to~$\pi$, i.e., the desired section.
\end{proof}

\begin{lemma}\label{vec.bundle}
Let $\pi:C\to Y$ be a cone over~$Y$, i.e., an affine morphism such that $\pi_*\Oo_C=\bigoplus_{n=0}^{\infty}\Cc_n$ is a positively graded $\Oo_Y$-algebra sheaf generated by $\Cc_0=\Oo_Y$ and~$\Cc_1$. Suppose that $C$ contains a closed subscheme $V$ which is a vector bundle over $Y$ and the fibers $C_y=V_y$ at some point $y\in Y$ coincide. Then $C=V$ over a neighborhood of~$y$.
\end{lemma}

\begin{proof}
There are $\Oo_Y$-module epimorphisms $\Cc_n\to\Sym_{\Oo_Y}^n\Vv^*$, where $\Vv$ is the sheaf of sections of $V$ over~$Y$. As $C_y=V_y$, these epimorphisms are isomorphisms at~$y$. Since $\Vv^*$ is locally free, they are isomorphisms in a neighborhood of $y$ by Nakayama's lemma.
\end{proof}

\begin{proposition}\label{fixed.pts}
There is a unique irreducible component $\Zz^{\circ}$ of $\Zz$ intersecting~$\So^{\gamma}$. It is smooth along $\So^{\gamma}$ and $\dim\Zz^{\circ}=\dim\So^{\gamma}+m$.
\end{proposition}

\begin{proof}
By Lemma~\ref{sec.cone}, $C(\Zz/\So^{\gamma})$ contains a closed subscheme $\Vv$ which is a trivial vector bundle over $\So^{\gamma}$ with trivializing sections~$\nabla{F_i}|_{\So^{\gamma}}$. We shall prove that $C(\Zz/\So^{\gamma})=\Vv$, which will imply all assertions.

In order to do this, we contract $\Zz$ to $\Zz_0$ inside the family ${\widehat\Zz\to\AAA^1}$. For $\Zz_0$ (i.e., in the case $M=T^*S$) the above claim is true. In particular, $C(\Zz_0/\So^{\gamma})=\Vv_0$ is the zero fiber of the family $\widehat\Vv\to\AAA^1$ of vector bundles defined like $\Vv$ by Hamiltonian vector fields~$\nabla\widehat{F}_i$ on $\So^{\gamma}\times\AAA^1$. By Lemma~\ref{vec.bundle}, $C(\widehat\Zz/\So^{\gamma}\times\AAA^1)=\widehat\Vv$ (here we use $\kk^{\times}$-equivariance to spread a neighborhood of $\So^{\gamma}\times\{0\}$ over the whole $\So^{\gamma}\times\AAA^1$), whence $C(\Zz/\So^{\gamma})=\Vv$.
\end{proof}

\begin{remark}
The geometric idea behind the construction of $\Zz^{\circ}$ is to spread $\So^{\gamma}$ along the trajectories of Hamiltonian vector fields~$\nabla{F_i}$. The above reasoning is an adaptation of this idea to algebraic geometry.
\end{remark}

Recall the following result of Bia{\l}ynicki-Birula:

\begin{theorem}[{\cite[Thm.\,4.1]{BB-dec}}]\label{BB-dec}
Let a multiplicative 1-parameter group $\gamma$ act on a smooth variety~$X$.
\begin{enumerate}
  \item The set of fixed points $X^{\gamma}$ is a smooth closed subset of~$X$.
  \item For any $y\in X^{\gamma}$ consider the decomposition
\begin{equation*}
    T_yX=T_y^+X\oplus T_y^0X\oplus T_y^-X
\end{equation*}
    into the sum of $\gamma$-eigenspaces of positive, zero, and negative eigenweights, respectively.
    Then $T_y(X^{\gamma})=T_y^0X$.
  \item For each irreducible (=connected) component $X_i$ of $X^{\gamma}$ the set
\begin{equation*}
      X_i^+=\{x\in{X}\mid\exists\lim_{t\to0}\gamma(t)x\in X_i\}.
\end{equation*}
    is a locally closed subvariety in $X$, and the map $X_i^+\to X_i$, $x\mapsto\lim_{t\to0}\gamma(t)x$ turns $X_i^+$
    into a locally trivial fibration over $X_i$ with fibers $(X_i^+)_y\simeq T^+_yX$ over $y\in X_i$.
\end{enumerate}
\end{theorem}

Now we define $\Ww=\{p\in M\mid\exists\lim_{t\to0}\gamma(t)p\in\Zz^{\circ}\}$. By Theorem~\ref{BB-dec}, $\Ww$~is a subvariety in $M$ (in fact, a locally trivial fibration into affine spaces over~$\Zz^{\circ}$). Since $\Zz^{\circ}$ is $(\overline{M}\cap P)$-stable and the conjugation by $\gamma(t)$ contracts $\overline{P}$ to $\overline{M}\cap P$ as $t\to0$, $\Ww$~is $\overline{P}$-stable. It can be included in a family $\widehat\Ww\to\AAA^1$ of subvarieties $\Ww_c=\widehat\Ww\cap\Mo_c$ defined in the same way. Note that $\Ww_0=\Uu$ is the conormal bundle of the foliation of $\overline{U}$-orbits in $\overline{U}(Z\cap S)$. Indeed, this stems from the $\gamma$-equivariant isomorphisms
\begin{align*}
\begin{split}
T^*\So &\simeq (\Ru[-]{\overline{Q}}\cap\Ru{P}) \times (\Ru{\overline{Q}}\cap\Ru{P}) \times (\overline{M}\cap\Ru{P}) \times T^*(Z\cap S) \times \strut\\
&\hphantom{\strut\simeq\strut} \times (\Ru{\overline\q}\cap\Ru[-]\p) \times (\Ru[-]{\overline\q}\cap\Ru[-]\p) \times (\overline\m\cap\Ru[-]\p),                                                                                         \end{split}
\\
\Uu &\simeq (\Ru{\overline{Q}}\cap\Ru{P}) \times (\overline{M}\cap\Ru{P}) \times T^*(Z\cap S) \times (\Ru{\overline\q}\cap\Ru[-]\p),
\\
\Zz_0 &\simeq (\overline{M}\cap\Ru{P}) \times T^*(Z\cap S)
\end{align*}
by observing that $\gamma$ acts on $\Ru{\overline{Q}}\cap\Ru{P}\simeq\Ru{\overline\q}\cap\Ru\p$ and $\Ru{\overline\q}\cap\Ru[-]\p$ with positive eigenweights, on $\Ru[-]{\overline{Q}}\cap\Ru{P}\simeq\Ru[-]{\overline\q}\cap\Ru\p$ and $\Ru[-]{\overline\q}\cap\Ru[-]\p$ with negative eigenweights, and trivially on $\overline{M}\cap\Ru{P}\simeq\overline\m\cap\Ru\p$, $\overline\m\cap\Ru[-]\p$, and $T^*(Z\cap S)$. (In the above isomorphisms we use the identifications $T^*\So\simeq\Ru{P}\times T^*\So|_{Z\cap S}$, $T^*_z\So\simeq T^*_z(Z\cap S)\oplus(\Ru{\p}z)^*$, and $(\Ru{\p}z)^*\simeq\Ru\p^*\simeq\Ru[-]\p$, $\forall z\in Z\cap S$.)

Theorem~\ref{mom(Lagr)} stems from the following proposition, together with Theorem~\ref{horosph}.

\begin{proposition}\label{def.horosph}
$\overline{\Phi(\Zz^{\circ})}=\ab\oplus(\overline\m\cap\Ru\p)$, $\overline{\Phi(\Ww)}=\overline\p_0^{\ann}$, and $G\Ww$ is dense in~$M$.
\end{proposition}

\begin{proof}
Recall that $\gamma$ acts on $\Ru[\pm]{\overline\q}$ with positive/negative eigenweights and on $\overline\m$ trivially.
First note that $\Phi(\Zz^{\circ})\subset\p_0^{\ann}\cap\overline\m=\ab\oplus(\overline\m\cap\Ru\p)$. Since $\gamma$ contracts $\Ww$ onto~$\Zz^{\circ}$, it contracts $\Phi(\Ww)$ to~$\ab\oplus(\overline\m\cap\Ru\p)$, whence $\overline{\Phi(\Ww)}\subset\ab\oplus(\overline\m\cap\Ru\p)\oplus\Ru\q=\overline\p_0^{\ann}$. The same holds for $\Ww_c$ instead of~$\Ww$. But $\overline{\Phi_0(\Ww_0)}=\overline\p_0^{\ann}$ by Theorem~\ref{horosph}, whence general fibers of $\Phi_0:\Ww_0\to\overline\p_0^{\ann}$ have dimension $\dim\Ww_0-\dim\overline\p_0^{\ann}$. By the fiber dimension theorem, general fibers of $\widehat\Phi:\widehat\Ww\to\overline\p_0^{\ann}\times\AAA^1$ and $\Phi:\Ww\to\overline\p_0^{\ann}$ have the same dimension, whence $\overline{\Phi(\Ww)}=\overline\p_0^{\ann}$. It then follows by $\gamma$-contraction that $\overline{\Phi(\Zz^{\circ})}=\overline\p_0^{\ann}\cap\overline\m=\ab\oplus(\overline\m\cap\Ru\p)$. The proof of $\overline{G\Ww}=M$ is the same as in Theorem~\ref{horosph}: the only facts which we use are $\overline{\Phi(\Ww)}=\overline\p_0^{\ann}$ and $\codim\Ww=\codim\Uu=\dim\Ru[-]{\overline{P}}$.
\end{proof}

\begin{remark}
If $M=L$ (e.g., $S$~is quasiaffine), then $\Ww=\Zz^{\circ}$ is the irreducible component of $\Phi^{-1}(\p_0^{\ann})$ intersecting~$\So$, and the proof of Theorem~\ref{mom(Lagr)} simplifies.
\end{remark}

\subsection{}

We conclude this section with a result on the $G$-action on the zero fiber
of the moment map.
%
%
%
%
Let $M$ be a Hamiltonian $G$-variety and $S\subset M$ be a closed $G$-stable smooth isotropic subvariety. We may assume that $\Phi(S)=\{0\}$.

\begin{definition}
The \emph{nullcone} of $M$ with respect to $S$ is the set
\begin{equation*}
\N=\{p\in M\mid\overline{Gp}\cap S\ne\emptyset\}.
\end{equation*}
\end{definition}

In the basic example where $M=T^*X$ for some smooth $G$-variety $X$ and $S$ is the zero section, $\N$~is a cone bundle over $X$ consisting of covectors which are contracted to $0$ by~$G$.

\begin{proposition}\label{nullcone}
If $M$ is affine, then (the smooth locus of) every irreducible component of $\Phi^{-1}(0)\cap\N$ is an isotropic subvariety.
\end{proposition}

\begin{proof}
Take any $p\in\N$. By standard facts of Invariant Theory \cite[4.4]{inv}, $\overline{Gp}$~contains a unique closed $G$-orbit $Gp_0\subset S$. By a theorem of Birkes--Richardson \cite[6.8]{inv}, there exists a 1-parameter subgroup $\gamma:\kk^{\times}\to G$ such that $\lim_{t\to0}\gamma(t)p\in Gp_0$. Moving $p,p_0$ inside their $G$-orbits, we may assume that $\gamma(t)\in T$ and $\lim_{t\to0}\gamma(t)p=p_0\in S^{\gamma}$. Furthermore, there are finitely many 1-parameter subgroups $\gamma_i:\kk^{\times}\to T$ such that $\N=\bigcup_iG\N_i$, where $\N_i=\{p\in M \mid
\exists\lim_{t\to0}\gamma_i(t)p\in S^{\gamma_i}\}$.

Let us prove that each irreducible component of each $\N_i$ is an isotropic subvariety. By Theorem~\ref{BB-dec}, $S^{\gamma_i}$~is smooth and $\N_i$ is a locally trivial fibration into affine spaces over~$S^{\gamma_i}$, whose tangent space at $p_0\in S^{\gamma_i}$ is $T_{p_0}\N_i=T_{p_0}^0S\oplus T_{p_0}^+M$. But this space is isotropic, because the symplectic form~$\omega$, being $\gamma_i$-invariant, can induce a nonzero pairing only between $\gamma_i$-eigenspaces of opposite eigenweights and $T_{p_0}^0S$ is isotropic.

Since $\gamma_i$ contracts $\N_i$ to~$S^{\gamma_i}$, the only $\gamma_i$-invariant forms on $\N_i$ are those pulled back from~$S^{\gamma_i}$. Indeed, the contraction map $\sigma:\N_i\to S^{\gamma_i}$, $p\mapsto\lim_{t\to0}\gamma_i(t)p$, is an affine morphism and $\sigma_*\Oo_{\N_i}$ is a negatively graded $\Oo_{S^{\gamma_i}}$-algebra sheaf containing $\Oo_{S^{\gamma_i}}$ as the graded component of degree~$0$. Hence $\sigma_*\Omega^{\bullet}_{\N_i}$ is negatively graded by the $\gamma_i$-action as well and its graded component of degree~$0$ is~$\Omega^{\bullet}_{S^{\gamma_i}}$, the sheaf of differential forms on~$S^{\gamma_i}$, as claimed. It follows that $T_p\N_i$ is isotropic for any $p\in\N_i$.

Hence the irreducible components of $\Phi^{-1}(0)\cap\N_i$ are isotropic subvarieties, too. We conclude by the following easy lemma:

\begin{lemma}
If $Z\subset\Phi^{-1}(0)$ is an isotropic subvariety, then $Y=\overline{GZ}$ is an isotropic subvariety as well.
\end{lemma}

A proof stems from a simple observation that $T_pY=\g{p}+T_pZ$ for general $p\in Z$ and $\omega(\xi{p},\nu)=\langle d\Phi_p(\nu),\xi\rangle=0$, $\forall\xi\in\g$, $\nu\in T_pY$, since $d\Phi$ vanishes on $Y\subset\Phi^{-1}(0)$.
\end{proof}


\section{Coisotropic subvarieties}

The aim of this section is to generalize Theorems \ref{c&r(Lagr)} and~\ref{mom(Lagr)} to a certain class of invariant coisotropic subvarieties in~$M$. Suppose that $S\subset M$ is a coisotropic subvariety. We retain the notation of \ref{deform} for this situation.

\begin{lemma}
$N\simeq(TS^{\sort})^*$.
\end{lemma}

A proof is the same as for Lemma~\ref{N=T*S}. Thus we have a natural surjective linear map $\psi:T^*S\to N$ of vector bundles over $S$ whose kernel is the annihilator of $(TS)^{\sort}\subset TS$.

The variety $\widehat{M}$ is equipped with a Poisson structure exactly in the same way as above. But $M_0$ is no longer a symplectic variety.

\begin{lemma}
The Poisson structure on $M_0=N$ descends from~$T^*S$, i.e., $\Oo_N\subset\psi_*\Oo_{T^*S}$ is closed under the Poisson bracket on $\Oo_{T^*S}$ and the restricted Poisson bracket on $\Oo_N$ coincides with~$\{\cdot,\cdot\}_0$.
\end{lemma}

A proof is the same as for Lemma~\ref{def.sympl}, taking into account that $\Nn^*\simeq\Tt_S^{\sort}$ is generated by~$\nabla{f}$, $f\in\Ii_S$.

\begin{definition}
A coisotropic subvariety $S\subset M$ is called \emph{special} if $\g{p}\subset(T_pS)^{\sort}$, $\forall p\in S$. A special coisotropic subvariety is automatically $G$-stable.
\end{definition}

For a special subvariety $S$ in a Hamiltonian $G$-variety~$M$, the moment map of $M$ can be shifted so that $\Phi(S)=0$. Indeed, $\langle d_p\Phi(\nu),\xi\rangle=\omega(\xi{p},\nu)=0$, $\forall p\in S$, $\nu\in T_pS$, $\xi\in\g$, whence $\Phi(S)$ is a $G$-fixed point. The construction of the total moment map and its properties extend word by word. Note that the moment map of $T^*S$ factors as~$\Phi_0\psi$.

\begin{theorem}\label{c&r&mom(coiso)}
Let $M$ be a Hamiltonian $G$-variety and $S\subset M$ be a special coisotropic subvariety. Then $2c(S)=\cork{M}+ 2\dim{S}-\dim{M}$, $r(S)=\df{M}$, and the closures of the images of the moment maps for $M$ and $T^*S$ coincide.
\end{theorem}

\begin{proof}
It goes along the same lines as the proofs of Theorems~\ref{c&r(Lagr)},\ref{mom(Lagr)}, with appropriate modifications.

We choose $P_0$-invariant functions $f_i$ on $\So$ and extend them to $P_0$-invariant functions $F_i$ on $\Mo$ as in~\ref{proof}. Then $df_i$ span the conormal bundle of the foliation of $U$-orbits in~$\So$, which contains the annihilator $\Ker\psi|_{\So}$ of $(T\So)^{\sort}$ by speciality. Shrinking~$\Mo$, we can choose $f_1,\dots,f_k$ such that $df_1,\dots,df_k$ freely generate this conormal bundle modulo~$\Ker\psi$ and forget about~$f_i$, $i>k$. Then, as before, $\nabla{F_i}$~are linearly independent on $\So$ modulo $T\So$ and $\nabla\widehat{F}_i$~are linearly independent on $\So\times\AAA^1$ modulo $T(\So\times\AAA^1)$ (here we use $\nabla\widehat{F}_i|_{M_0}=-df_i|_{(T\So)^{\sort}}$).

Now we define subschemes $\Zz$ and $\widehat\Zz$ as in~\ref{proof}, and observe that $\psi^{-1}(\Zz_0)\simeq(\overline{M}\cap\Ru{P})\times T^*(Z\cap S)\subset\Ru{P}\times T^*(Z\cap S)\times\Ru[-]\p\simeq T^*\So$. Similarly to Proposition~\ref{fixed.pts} we prove that $\Zz$ contains a unique irreducible component $\Zz^{\circ}$ intersecting~$\So^{\gamma}$, which is smooth along~$\So^{\gamma}$, and $\dim\Zz^{\circ}=\dim\So^{\gamma}+k$.

After that, we define $\Ww$ and $\widehat\Ww$ as in~\ref{proof}, and observe that $\psi^{-1}(\Ww_0)=\Uu$. Indeed,
\begin{equation*}
\exists\lim_{t\to0}\gamma(t)\psi(q)\in\Zz_0\iff\exists\lim_{t\to0}\gamma(t)q\in\psi^{-1}(\Zz_0),
\end{equation*}
because $\Ker\psi|_{\So}\subset\Ru{P}\times T^*(Z\cap S)$ and $\gamma$ leaves $T^*(Z\cap S)$ pointwise fixed. Proposition~\ref{def.horosph}, together with the proof, is extended to our situation word by word and completes the proof of the theorem.
\end{proof}

\begin{corollary}
General $G$-orbits in $M$ and $N$ have the same dimension.
\end{corollary}

\begin{proof}


We recall that in a symplectic variety $M$ we have $\dim Gx = \dim\overline{\Phi(M)}$ for general $x\in M$. For a Poisson variety $N$ this is not so obvious: one has to consider the image of the moment map on a general symplectic leaf in $N$ instead. However it turns out, due to the structure of the moment map for~$T^*S$, that all these images have one and the same closure.

Specifically, let $p\in N$ be a general point. We have $p\in N_y$, where $y\in S$ is a general point. Take a general point $q\in\psi^{-1}(p)$. The moment map $\Phi:T^*S|_{Gy}\to\g^*$ factors as $\Phi=\overline\Phi\rho$, where $\overline\Phi:T^*(Gy)\to\g^*$ is the moment map and $\rho:T^*S|_{Gy}\to T^*(Gy)$ is the restriction map. On the other hand, since $\Phi=\Phi_0\psi$, the moment map $\Phi_0:N|_{Gy}\to\g^*$ factors as $\Phi_0=\overline\Phi\overline\rho$, where $\overline\rho:N|_{Gy}\to T^*(Gy)$ is the restriction of linear functions from $N^*_z\simeq(T_zS)^{\sort}$ to $\g{z}\subset(T_zS)^{\sort}$, $\forall z\in Gy$, which yields a splitting $\rho=\overline\rho\psi$. Put $\overline{q}=\rho(q)=\overline\rho(p)$. By Theorem~\ref{horosph}, (the closure of) the image of the moment map for a cotangent bundle is determined by the local structure of the underlying variety, specifically, by the normalizers of general $B$- and $U$-orbits therein, whence the closures of $\Phi(T^*S)$ and $\overline\Phi(T^*Gy)$ coincide. It follows that $\dim Gx = \dim Gq = \dim G\overline{q} = \dim Gp$.
\end{proof}

\end{document}